\documentclass[12pt]{amsart}
\usepackage{amsfonts}
\usepackage{amssymb}
\usepackage{enumerate}
\usepackage{graphicx}

\makeatletter
\@namedef{subjclassname@2010}{  \textup{2010} Mathematics Subject Classification}
\makeatother
\newtheorem{theorem}{Theorem}[section]

\newtheorem{corollary}[theorem]{Corollary}

\theoremstyle{definition}
\newtheorem{definition}[theorem]{Definition}

\newtheorem{remark}[theorem]{Remark}

\DeclareMathOperator{\diam}{diam}
\DeclareMathOperator{\co}{\overline{co}}

\begin{document}
\title[A fixed point theorem in $B(H,\ell _{\infty })$]{A fixed point
theorem in $B(H,\ell _{\infty })$}
\author[A. Wi\'{s}nicki]{Andrzej Wi\'{s}nicki}
\address{Pedagogical University of Krak\'{o}w, Cracow, Poland}
\email{andrzej.wisnicki@up.krakow.pl}
\date{}

\begin{abstract}
We show that if $X$ is a complete metric space with uniform relative normal
structure and $G$ is a subgroup of the isometry group of $X$ with bounded
orbits, then there is a point in $X$ fixed by every isometry in $G$. As a
corollary, we obtain a theorem of U. Lang (2013) concerning injective metric
spaces. A few applications of this theorem are given to the problems of
inner derivations. In particular, we show that if $L_{1}(\mu )$ is an
essential Banach $L_{1}(G)$-bimodule, then any continuous derivation $\delta
:L_{1}(G)\rightarrow L_{\infty }(\mu )$ is inner. This extends a theorem of
B. E. Johnson (1991) asserting that the convolution algebra $L_{1}(G)$ is
weakly amenable if $G$ is a locally compact group.
\end{abstract}

\maketitle





\section{Introduction}

The writing of this paper has been animated by a fixed point theorem of
Bader, Gelander and Monod in \cite{BGM}--there is a point fixed by every
isometry of an $L$\textit{-embedded} Banach space that preserves a given
bounded set. One of its consequences is the optimal solution to the
following \textit{derivation problem}: if $G$ is a locally compact group,
then any derivation from the convolution algebra $L_{1}(G)$ to $M(G)$ is 
\textit{inner}. The problem was studied since 1960s and proved for the first
time by Losert \cite{Lo}. An ingenious proof in \cite{BGM} relies on the
Ryll-Nardzewski theorem. Another consequence of this theorem, observed by
Uffe Haagerup, is a direct proof of \textit{weak amenability} of all C$%
^{\ast }$-algebras which was proved in \cite{Ha} using the
Grothendieck--Haagerup--Pisier inequality.

There is a counterpart of the Bader--Gelander--Monod theorem in \textit{%
injective} Banach spaces $L_{\infty }$ (see Section 2), and it is natural to
ask about the noncommutative analogue of that result. The relevance of such
a question is particularly apparent if we analyse the proof of Corollary C
in \cite{BGM}. The central idea behind the proofs of both the derivation
problem and weak amenability of C$^{\ast }$-algebras there, is to construct
a group of affine isometries with bounded orbits acting on the predual of an
appropriate von Neumann algebra and then applying \cite[Theorem A]{BGM}. We
can use a similar strategy to study derivations into injective Banach
bimodules and, more general, into Banach bimodules with \textit{uniform
relative normal structure} (see Section 2 for the definition), based this
time on Theorem \ref{La}. In particular, we show that if $L_{1}(\mu )$ is an
essential Banach $L_{1}(G)$-bimodule, then any continuous derivation $\delta
:L_{1}(G)\rightarrow L_{\infty }(\mu )$ is inner, where $\mu $ is a $\sigma $%
-finite measure on $(E,\Sigma )$ or a Radon measure on a locally compact
space. This extends a theorem of Johnson (see \cite[Theorem]{Jo1}) asserting
that the convolution algebra $L_{1}(G)$ is weakly amenable if $G$ is a
locally compact group.

In Section 3 we apply Theorem \ref{La} in the case of $B(H,\ell _{\infty
}(\Gamma ))$, the Banach space of bounded operators from a Hilbert space $H$
into $\ell _{\infty }(\Gamma )$ for some $\Gamma $. It can throw some new
light on the \textit{Kadison similarity problem. }One of the equivalent
formulations of this problem concerns inner derivations into Banach modules
and translates into the existence or non-existence of a fixed point for the
appropriate group of affine isometries acting on $B(H)$, the von Neumann
algebra of bounded linear operators on a Hilbert space $H$. Notice that $%
B(H) $ is injective in the category of operator spaces and completely
contractive maps. Though it is not injective as a Banach space, an embedding
of $B(H)$ into a Banach space $B(H,\ell _{\infty }(\Gamma ))$ with uniform
relative normal structure seems to lead to something new.

Our strategy is as follows. Let $\mathcal{A}$ be a unital C$^{\ast }$%
-algebra, $M$ a Banach bimodule over $\mathcal{A}$ and $\delta :\mathcal{A}%
\rightarrow M$ a derivation: $\delta (ab)=\delta (a)b+a\delta (b).$ If $G$
is a group of unitary elements of $\mathcal{A}$, then $\alpha :G\times
M\rightarrow M$ defined by $\alpha (g)x=gxg^{-1}+\delta (g)g^{-1}$, $g\in
G,x\in M,$ is an affine and isometric action of $G$ on $M.$ It follows from
Theorem \ref{La} that if $M$ is a Banach bimodule with uniform relative
normal structure, then there is $x\in M$ such that $\alpha (g)x=x$, and
consequently, $\delta (g)=xg-gx$. Since every element of $\mathcal{A}$ is a
linear combination of three unitaries, we deduce that each derivation $%
\delta :\mathcal{A}\rightarrow M$ is inner. The argument works also for some
group algebras, see Theorem \ref{group}. The case of the convolution algebra 
$L_{1}(G)$ and $M=L_{\infty }(G)$, where $G$ is a locally compact group was
studied by Johnson \cite{Jo2, Jo1}. Let us move on now to the Kadison
similarity problem, which is equivalent to the derivation problem for
derivations $\delta :B(H)\supset \mathcal{A}\rightarrow B(H)$ (see \cite{Ki}%
). Let us embed a Hilbert space $H$ isometrically in an injective space $%
\ell _{\infty }(\Gamma )$ and extend each unitary $g:H\rightarrow H$ to an
isometry $\tilde{g}:\ell _{\infty }(\Gamma )\rightarrow \ell _{\infty
}(\Gamma )$. Now, consider the group of isometries $\alpha _{g}:B(H,\ell
_{\infty }(\Gamma ))\rightarrow B(H,\ell _{\infty }(\Gamma ))$ defined by $%
\alpha _{g}x=\tilde{g}xg^{-1}+\delta (g)g^{-1}$. It follows from Theorem \ref%
{main} that there is a bounded linear operator $T:H\rightarrow \ell _{\infty
}(\Gamma )$ such that $\delta (g)=Tg-\tilde{g}T$. Note that in fact a
similar result holds for derivations from $\ell _{1}(G)$ to $B(H)$ on which $%
G$ acts by isometries. However, in the case of C$^{\ast }$-algebras our
factorization appears to be more effective. For example, properties of the
group homomorphism $g\mapsto \tilde{g}$ could be useful in studying complete
boundedness of $\delta $. Recall that Christensen's theorem \cite{Ch2}
asserts that a derivation $\delta :\mathcal{A}\rightarrow B(H)$ is inner if
and only if it is completely bounded whenever $\mathcal{A\,}$\ is a unital C$%
^{\ast }$-algebra. Thus we can reduce the Kadison similarity problem to the
problem of complete boundedness of the appropriate homomorphism from $%
\mathcal{A}$ into $B(Y)$ for some space $Y$ isomorphic to a Hilbert space
(see Remark \ref{rem}).

\section{Spaces with uniform relative normal structure}

Let $\mathcal{A}(X)$ denote the family of admissible subsets of a metric
space $X$, i.e., the intersections of closed balls of $X$. The following
definition, originated in the works of Brodski\u{\i} and Mil'man \cite{BrMi}%
, and Kirk \cite{Kir}, was introduced by Soardi \cite{So} in the case of
Banach spaces (see also \cite{KhKi}).

\begin{definition}
A metric space $X$ is said to have uniform relative normal structure if
there exists $c<1$ such that for each nonempty and bounded set $M\in 
\mathcal{A}(X)$ there exists $z_{M}\in X$ such that
\end{definition}

\begin{itemize}
\item $d(x,z_{M})\leq c%
\diam%
M$ for every $x\in M;$

\item if for some $y\in X$, $M\subset B(y,c%
\diam%
M)$, then 
\begin{equation*}
d(y,z_{M})\leq c%
\diam%
M,
\end{equation*}
\end{itemize}

\noindent where $B(y,r)$ denotes the closed ball centered at $y$ and of
radius $r$.

Examples of spaces with uniform relative normal structure include Hilbert
spaces, ($1$-)injective spaces and their closed balls (see \cite{So}). It
was proved in \cite{So} that if $K$ is a nonempty weak$^{\ast }$-closed
subset with uniform relative normal structure of a dual Banach space, and $%
T:K\rightarrow K$ is nonexpansive and leaves invariant a weak$^{\ast }$%
-compact subset of $K$, then $K$ contains a fixed point for $T$ (see also 
\cite{EJS, KhKi} for some generalizations in metric spaces). Lau \cite{Lau}
used Soardi's arguments to show the following fixed point theorem for
isometries.

\begin{theorem}[A.T.-M. Lau]
\label{La0}Let $X$ be a Banach space and let $K$ be a nonempty bounded
closed subset of $X$ with uniform relative normal structure. Then there is $%
z\in K$ such that $Tz=z$ for each isometry $T$ from $K$ onto $K$.
\end{theorem}

A natural question arises whether it is possible to extend the above results
for mappings with bounded orbits, and Prus (see \cite{EsKh}) constructed an
isometry $T:\ell _{\infty }\rightarrow \ell _{\infty }$ with bounded orbits
but without fixed points. However, it turns out that in the case of
invertible isometries the situation is different, and the Soardi--Lau method
leaves some room for improvement.

\begin{theorem}
\label{La}Suppose that $X$ is a complete metric space with uniform relative
normal structure. If $G$ is a group of isometries $g:X\rightarrow X$ with
bounded orbits, then there exists $x\in X$ such that $gx=x$ for every $g\in
G.$
\end{theorem}

\begin{proof}
For any admissible set $M\subset X$, define%
\begin{eqnarray*}
A(M) &=&\bigcap_{x\in M}B(x,c%
\diam%
M), \\
H(M) &=&\bigcap_{y\in A\left( M\right) }B\left( y,c%
\diam%
M\right) \cap A\left( M\right) ,
\end{eqnarray*}%
where $c<1$ satisfies the conditions of uniform relative normal structure
for $X.$ Notice that $H(M)\neq \emptyset $ because $z_{M}\in H(M).$

Fix $x_{0}\in X$ and let $A_{0}=\bigcap_{g\in G}B(gx_{0},\delta (x_{0}))$,
where

\begin{equation*}
\delta (x_{0})=\sup_{g,h\in G}d\left( gx_{0},hx_{0}\right)
\end{equation*}
denotes the diameter of the orbit of $x_{0}$. Note that $A_{0}$ is
admissible and%
\begin{equation*}
g(A_{0})=\bigcap_{h\in G}B(ghx_{0},\delta (x))=A_{0}.
\end{equation*}

We take by recursion $A_{1}=H(A_{0}),A_{n+1}=H(A_{n})$, and pick $x_{n}\in
A_{n}$ for $n=1,2,\ldots $. If $z\in A_{n+1}$, then $d(z,y)\leq c%
\diam%
A_{n}$ for each $y\in A_{n+1}\subset A(A_{n})$ and hence%
\begin{equation*}
\diam%
A_{n+1}\leq c%
\diam%
A_{n}\leq ...\leq c^{n}%
\diam%
A_{1}.
\end{equation*}%
Similarly, since $x_{n}\in A(A_{n})$,%
\begin{equation*}
d(x_{n+1},x_{n})\leq c%
\diam%
A_{n}\leq ...\leq c^{n}%
\diam%
A_{1}.
\end{equation*}

Thus $(x_{n})$ is a Cauchy sequence converging to some $x\in X.$ Moreover, $%
g(A_{n})=A_{n}$ for every $g\in G.$ It follows that 
\begin{equation*}
d(gx,x)\leq d(gx,gx_{n})+d(gx_{n},x_{n})+d(x_{n},x)\longrightarrow 0\text{
if }n\rightarrow \infty ,
\end{equation*}%
and hence $gx=x$ for every $g\in G.$
\end{proof}

In particular, if $X=K$ is a bounded closed subset of a Banach space we
obtain the Lau theorem \ref{La0}, and Bruhat--Tits' theorem \cite{BrTi} (see
also \cite[Corollary II.2.8]{BrHa}) follows if we consider a closed convex
subset of a Hadamard space. But the applications of Theorem \ref{La} are
even wider. Recall that a metric space is injective (or hyperconvex) if 
\begin{equation*}
\bigcap_{\alpha \in I}B(x_{\alpha },r_{\alpha })\neq \emptyset
\end{equation*}%
for any collection of closed balls $\{B(x_{\alpha },r_{\alpha })\}_{\alpha
\in I}$ such that $d(x_{\alpha },x_{\beta })\leq r_{\alpha }+r_{\beta },\
\alpha ,\beta \in I.$ Notice that injective metric spaces have uniform
relative normal structure with $c=1/2$, and thus we obtain the following
corollary proved by Lang \cite[Proposition 1.2]{La} with completely
different methods (see also \cite[Theorem 2.3]{WiWo} for a little more
general result).

\begin{corollary}[U. Lang]
\label{Lang}Let $X$ be an injective metric space. If $G$ is a group of
isometries $g:X\rightarrow X$ with bounded orbits, then there is a fixed
point of $G$ in $X$.
\end{corollary}

In particular, Corollary \ref{Lang} applies to a group of isometries with
bounded orbits acting on a real $L_{\infty }(\mu )$ whenever $\mu $ is a $%
\sigma $-finite measure on $(E,\Sigma )$, a Radon measure on a locally
compact space (see \cite[Corollary 3.10]{DeFl}), or a counting measure on
any $\Gamma $. Note that a complex $L_{\infty }(\mu )$ need not be
hyperconvex, yet it follows from \cite[Lemma]{So} that it has uniform
relative normal structure with $c=1/\sqrt{2}.$

\begin{corollary}
Let $\mu $ be a $\sigma $-finite measure on $(E,\Sigma )$, a Radon measure
on a locally compact space, or a counting measure on any $\Gamma $. If $G$
is a subgroup of the isometry group of a real or complex $L_{\infty }(\mu )$
with bounded orbits, then there is a fixed point of $G$ in $L_{\infty }(\mu
) $.
\end{corollary}

Equipped with Theorem \ref{La}, the discussion around Corollary C in \cite%
{BGM} can be repeated verbatim for the case of Banach bimodules with the
uniform relative normal structure.

If $\mathcal{A}$ is a Banach algebra [with unit $e$], a Banach space $M$ is
said to be a Banach $\mathcal{A}$-bimodule if there are bounded bilinear
mappings $(a,x)\mapsto a.x,(x,a)\mapsto x.a:\mathcal{A}\times M\rightarrow M$
such that [$e.x=x.e=x$ for every $x\in M$] and the usual associative law
holds for each type of triple product: $a_{1}a_{2}x,a_{1}xa_{2},xa_{1}a_{2}.$
We will assume that%
\begin{equation}
\left\Vert a.x\right\Vert \leq \left\Vert a\right\Vert \left\Vert
x\right\Vert ,\ \ \ \left\Vert x.a\right\Vert \leq \left\Vert x\right\Vert
\left\Vert a\right\Vert \ \ \ (a\in A,x\in M).  \label{iso}
\end{equation}%
Note that for any Banach $\mathcal{A}$-bimodule $M$ we can define an
equivalent norm which satisfies (\ref{iso}).

By a derivation, from a Banach algebra $\mathcal{A}$ into a Banach $\mathcal{%
A}$-bimodule $M$, we mean a linear mapping $\delta :\mathcal{A}\rightarrow M$
such that $\delta (ab)=\delta (a).b+a.\delta (b)$ for $a,b\in \mathcal{A}$.
A derivation is inner if there exists $T\in M$ such that $\delta (a)=T.a-a.T$
for each $a\in \mathcal{A}.$

\begin{theorem}
\label{Cstar}Let $\mathcal{A}$ be a unital $C^{\ast }$-algebra, and let $M$
be a Banach space with uniform relative normal structure. If $M$ is a Banach 
$\mathcal{A}$-bimodule, then any derivation $\delta :\mathcal{A}\rightarrow
M $ is inner.
\end{theorem}

\begin{proof}
Let $U(\mathcal{A})$ denotes the group of unitaries of $\mathcal{A}$. For
any $g\in U(\mathcal{A})$ and $x\in M$, set%
\begin{equation*}
\alpha (g)x=g.x.g^{-1}+\delta (g).g^{-1},
\end{equation*}%
and notice that $\alpha :U(\mathcal{A})\times M\rightarrow M$ is an (affine)
and isometric action of $U(\mathcal{A})$ on $M$, and thus%
\begin{eqnarray*}
\alpha (gh)x &=&\alpha (g)\alpha (h)x, \\
\left\Vert \alpha (g)x-\alpha (g)y\right\Vert &=&\left\Vert x-y\right\Vert \
\ \ (g,h\in U(\mathcal{A}),x,y\in M).
\end{eqnarray*}%
Furthermore, Ringrose's theorem \cite{Ri} yields $\delta :\mathcal{A}%
\rightarrow M$ is bounded, and hence the orbit%
\begin{equation*}
A=\{\delta (g).g^{-1}:g\in U(\mathcal{A})\}
\end{equation*}%
of $\alpha $ at $0$ is bounded too. It follows from Theorem \ref{La} that
there is $T\in M$ such that $\alpha (g)T=T$, which gives 
\begin{equation*}
\delta (g)=T.g-g.T\ \ \ (g\in U(\mathcal{A})).
\end{equation*}%
This finishes the proof since any element of $\mathcal{A}$ is a linear
combination of four unitaries (in fact, three \cite{KaPe}).
\end{proof}

We now turn to the case of group algebras $L_{1}(G)$, where $G$ is a locally
compact group. This case is a little more subtle. If $E$ is an essential
Banach $L_{1}(G)$-bimodule, i.e., the linear hull of $\{a.t.b:a,b\in
L_{1}(G),t\in E\}$ is dense in $E$, then the dual actions%
\begin{equation*}
\left\langle t,x.a\right\rangle =\left\langle a.t,x\right\rangle ,\ \ \
\left\langle t,a.x\right\rangle =\left\langle t.a,x\right\rangle \ \ \ (a\in
L_{1}(G),t\in E,x\in M)
\end{equation*}%
of $L_{1}(G)$ on $M=E^{\ast }$ can be extended to actions of $M(G)$, the
space of (complex-valued) Radon measures on $G$. Let $\delta
:L_{1}(G)\rightarrow M$ be a continuous derivation. Then there is a unique
derivation $\tilde{\delta}:M(G)\rightarrow M$ extending $\delta $ such that $%
\parallel \tilde{\delta}\parallel =\parallel \delta \parallel $, and 
\begin{equation*}
\langle a.t,\tilde{\delta}\mu \rangle =\left\langle t,\delta (\mu \ast
a)\right\rangle -\left\langle t.\mu ,\delta a\right\rangle \ \ \ (a\in
L_{1}(G),\mu \in M(G),t\in E),
\end{equation*}%
where $\ast $ denotes the convolution (with respect to a fixed left Haar
measure on $G$) and $t.\mu $ $(t\in E,\mu \in M(G))$ extends the action of $%
L_{1}(G)$ on $E$ to $M(G)$ (see \cite[Theorem 5.6.34]{Da}, \cite{Jo}). For
any $g\in G$ and $x\in M$, define an action $\alpha :G\times M\rightarrow M$
by setting%
\begin{equation*}
\alpha (g)x=\delta _{g}.x.\delta _{g^{-1}}+\tilde{\delta}(\delta
_{g}).\delta _{g^{-1}},
\end{equation*}%
where $\delta _{g}$ denotes the Dirac measure at $g.$ Since the dual Banach $%
L_{1}(G)$-bimodule $M$ satisfies (\ref{iso}), $\alpha $ is an isometric
action of $G$ on $M$ with bounded orbits, and thus by Theorem \ref{La},
there is $T\in M$ such that $\alpha (g)T=T$ whenever $M$ has uniform
relative normal structure. It follows that 
\begin{equation*}
\tilde{\delta}(\delta _{g})=T.\delta _{g}-\delta _{g}.T,
\end{equation*}%
and then Theorem 5.6.39 in \cite{Da} shows that $\delta :L_{1}(G)\rightarrow
M$ is inner. We are thus led to the following theorem.

\begin{theorem}
\label{group}Let $G$ be a locally compact group, and let $E$ be an essential
Banach $L_{1}(G)$-bimodule. If the dual $L_{1}(G)$-bimodule $E^{\ast }$ has
uniform relative normal structure, then any continuous derivation $\delta
:L_{1}(G)\rightarrow E^{\ast }$ is inner.
\end{theorem}

\begin{corollary}
\label{groupcor}Let $G$ be a locally compact group, and let $L_{1}(\mu )$ be
an essential Banach $L_{1}(G)$-bimodule, where $\mu $ is a $\sigma $-finite
measure on $(E,\Sigma )$ or a Radon measure on a locally compact space. Then
any continuous derivation $\delta :L_{1}(G)\rightarrow L_{\infty }(\mu )$
(into the dual $L_{1}(G)$-bimodule $L_{\infty }(\mu )$) is inner.
\end{corollary}

In particular, Corollary \ref{groupcor} applies to $L_{1}(G)$ itself with
the convolution product and in this case gives the theorem of Johnson \cite%
{Jo1}.

\section{A fixed point theorem in $B(H,\ell _{\infty })$}

In 1955, motivated by work of Dixmier and Day on uniformly bounded group
representations, Kadison \cite{Ka1} raised the problem of whether every
bounded homomorphism $u$ from C$^{\ast }$-algebra $\mathcal{A}$ into $B(H)$
is similar to a $\ast $-homomorphism, i.e., does there exist an invertible
operator $S\in B(H)$ such that the map $\tilde{u}(x)=S^{-1}u(x)S$ satisfies
the condition $\tilde{u}(x^{\ast })=\tilde{u}(x)^{\ast }$? The Kadison
similarity problem is equivalent to a number of other important problems
(see \cite{Cam, Ch2, HaPa, Pi, Pi1}), including the derivation problem \cite%
{Ki}--given a C$^{\ast }$-algebra $\mathcal{A}\subset B(H)$, is every
derivation $\delta :\mathcal{A}\rightarrow B(H)$ inner?

Our observation concerning derivations from a C$^{\ast }$-algebra $\mathcal{A%
}$ into $B(H)$ is based on the following consequence of Theorem \ref{La} in
Section 2.

\begin{theorem}
\label{main}Let $B(H,\ell _{\infty }(\Gamma ))$ be a Banach space of bounded
operators from a (complex) Hilbert space $H$ into a (complex) $\ell _{\infty
}(\Gamma )$ for some $\Gamma $. If $G$ is a group of isometries $g:B(H,\ell
_{\infty }(\Gamma ))\rightarrow B(H,\ell _{\infty }(\Gamma ))$ with bounded
orbits, then there exists $x\in B(H,\ell _{\infty }(\Gamma ))$ such that $%
gx=x$ for every $g\in G.$
\end{theorem}

\begin{proof}
Since $B(H,\ell _{\infty }(\Gamma ))$ is isometric to $\ell _{\infty
}(\Gamma ,H^{\ast })\simeq \ell _{\infty }(\Gamma ,H)$ by setting $B(H,\ell
_{\infty }(\Gamma ))\ni F\mapsto f\in \ell _{\infty }(\Gamma ,H^{\ast })$,
where 
\begin{equation*}
\left\langle h,f(\gamma )\right\rangle =F(h)(\gamma )\ \ \ (h\in H,\gamma
\in \Gamma ),
\end{equation*}%
it suffices to show, using Theorem \ref{La}, that $\ell _{\infty }(\Gamma ,H)
$ has uniform relative normal structure. It is quite simple (see also \cite[%
Proposition]{So}). For a bounded set $M\subset \ell _{\infty }(\Gamma ,H)$,
let 
\begin{equation*}
M_{i}=\{x_{i}\in H:(x_{\gamma })_{\gamma \in \Gamma }\in M\}\ \ \ (i\in
\Gamma ).
\end{equation*}%
Then, for each $i$, there exists $z_{i}\in \co M_{i}$ such that 
\begin{equation*}
\left\Vert x_{i}-z_{i}\right\Vert \leq \frac{\sqrt{3}}{2}%
\diam%
M_{i}\leq \frac{\sqrt{3}}{2}%
\diam%
M\ \ \ (x_{i}\in M_{i},i\in \Gamma ).
\end{equation*}%
Now suppose that $M\subset B((y_{\gamma }),\frac{\sqrt{3}}{2}%
\diam%
M)$ for some $(y_{\gamma })\in \ell _{\infty }(\Gamma ,H).$ Then $%
M_{i}\subset B(y_{i},\frac{\sqrt{3}}{2}%
\diam%
M)$ for each $i$, and hence also $\co M_{i}\subset B(y_{i},\frac{\sqrt{3}}{2}%
\diam%
M)$. It follows that 
\begin{equation*}
\left\Vert z_{i}-y_{i}\right\Vert \leq \frac{\sqrt{3}}{2}%
\diam%
M
\end{equation*}%
for every $i\in \Gamma $, which completes the proof.
\end{proof}

\begin{remark}
The method of proof carries over to $\ell _{\infty }(\Gamma ,X)$ when $X$ is
uniformly convex (see \cite{So}), and more generally when $X$ has uniform
normal structure (see \cite[Theorem VI.2.2]{ADL}).
\end{remark}

Given a Hilbert space $H$, we can embed it (linearly) isometrically into $%
\ell _{\infty }(\Gamma ),$ where $\Gamma =B_{H}$ (or any norming subset of $%
B_{H}$) by setting%
\begin{equation*}
H\ni h\mapsto (\left\langle h,\gamma \right\rangle )_{\gamma \in \Gamma }\in
\ell _{\infty }(\Gamma ).
\end{equation*}%
Let $\tilde{H}$ denote the isometric copy of $H$ in $\ell _{\infty }(\Gamma
) $, and define%
\begin{equation*}
\kappa (a)(\left\langle h,\gamma \right\rangle )_{\gamma \in \Gamma
}=(\left\langle ah,\gamma \right\rangle )_{\gamma \in \Gamma }\ \ \ (a\in
B(H),(\left\langle h,\gamma \right\rangle )_{\gamma \in \Gamma }\in \tilde{H}%
).
\end{equation*}%
Then $\kappa :B(H)\rightarrow B(\tilde{H})$ is a linear, isometric, unital $%
\ast $-isomorphism, i.e., $\kappa (ab)=\kappa (a)\circ \kappa (b)$ and $%
\kappa (a^{\ast })=\kappa (a)^{\ast }$ for every $a,b\in B(H)$. From now on
we will regard $H$ as a subspace of $\ell _{\infty }(\Gamma )$, and $B(H)$
as a subspace of $B(H,\ell _{\infty }(\Gamma ))$--the Banach space of
bounded linear operators from $H$ into $\ell _{\infty }(\Gamma )$.

Now let $\mathcal{A}\subset B(H)\subset B(H,\ell _{\infty }(\Gamma ))$ be a
unital C$^{\ast }$-algebra and extend each unitary $g\in U(\mathcal{A})$ to
a linear isometry $\tilde{g}:\ell _{\infty }(\Gamma )\rightarrow \ell
_{\infty }(\Gamma )$ by setting%
\begin{equation*}
\ell _{\infty }(\Gamma )\ni (z_{\gamma })_{\gamma \in \Gamma }\mapsto
(z_{g^{\ast }\gamma })_{\gamma \in \Gamma }\in \ell _{\infty }(\Gamma ).
\end{equation*}%
Then 
\begin{equation*}
(\tilde{g}\circ \tilde{h})(z_{\gamma })_{\gamma \in \Gamma }=\tilde{g}%
((z_{h^{\ast }\gamma })_{\gamma \in \Gamma })=(z_{h^{\ast }g^{\ast }\gamma
})_{\gamma \in \Gamma }=\widetilde{(g\circ h)}(z_{\gamma })_{\gamma \in
\Gamma }
\end{equation*}%
for every $g,h\in U(\mathcal{A}),(z_{\gamma })_{\gamma \in \Gamma }\in \ell
_{\infty }(\Gamma )$ and $\gamma \in \Gamma $. In other words, $%
(g,(z_{\gamma })_{\gamma \in \Gamma })$ $\mapsto \tilde{g}((z_{\gamma
})_{\gamma \in \Gamma })$ is a left action of $U(\mathcal{A})$ on $\ell
_{\infty }(\Gamma )$.

We are thus led to the following factorization theorem.

\begin{theorem}
\label{factor}Let $\mathcal{A}\subset B(H)\subset B(H,\ell _{\infty }(\Gamma
))$ be a unital C$^{\ast }$-algebra, and let $\delta :\mathcal{A}\rightarrow
B(H)$ be a derivation (with the usual product operation). Then there exists
a bounded linear operator $T:H\rightarrow \ell _{\infty }(\Gamma )$ such
that $\delta (g)=Tg-\tilde{g}T$ for every $g\in U(\mathcal{A})$.
\end{theorem}

\begin{proof}
The proof follows the arguments in Theorem \ref{Cstar}. For any unitary $%
g\in U(\mathcal{A})$ and $x\in X=B(H,\ell _{\infty }(\Gamma ))$, set%
\begin{equation*}
\alpha (g)x=\tilde{g}xg^{-1}+\delta (g)g^{-1},
\end{equation*}%
and note that $\alpha :U(\mathcal{A})\times X\rightarrow X$ is an isometric
action of $U(\mathcal{A})$ on $X$ with bounded orbits. Hence Theorem \ref%
{main} yields $\alpha (g)T=T$ for some $T\in X$. It follows that%
\begin{equation*}
\delta (g)=Tg-\tilde{g}T
\end{equation*}%
for each unitary element $g$, which completes the proof.
\end{proof}

\begin{remark}
\label{rem}Note that a similar result holds for derivations from $\ell
_{1}(G)$ to $B(H)$ on which $G$ acts by isometries, and it is well known
that there are non-inner derivations from $\ell _{1}(F_{2})$ to $B(\ell
_{2}(F_{2}))$ (see, e.g., \cite{Lo2}). However, in the case of C$^{\ast }$%
-algebras, the Christensen theorem (see \cite{Ch2}) asserts that a
derivation $\delta :\mathcal{A}\rightarrow B(H)$ is inner if and only if it
is completely bounded, and we can use Theorem \ref{factor} to study complete
boundedness of $\delta $. Let $T:H\rightarrow \ell _{\infty }(\Gamma )$ be a
bounded linear operator such that $\delta (g)=Tg-\tilde{g}T$ for every
unitary $g\in U(\mathcal{A})$. Following \cite{Pa}, define 
\begin{equation*}
S=\left( 
\begin{array}{cc}
I & T \\ 
0 & I%
\end{array}%
\right) :H\oplus _{2}H\rightarrow \ell _{\infty }(\Gamma )\oplus _{2}\ell
_{\infty }(\Gamma ),
\end{equation*}%
where $X=\ell _{\infty }(\Gamma )\oplus _{2}\ell _{\infty }(\Gamma )$
denotes a Banach space with the norm $\left\Vert (x_{1},x_{2})\right\Vert
=(\left\Vert x_{1}\right\Vert ^{2}+\left\Vert x_{2}\right\Vert ^{2})^{1/2}$.
Clearly, $H\oplus _{2}H$ is a Hilbert subspace of $X$, and note that $S$ is
an invertible operator onto $Y=S(H\oplus _{2}H)$ with the inverse $%
S^{-1}=\left( 
\begin{array}{cc}
I & -T \\ 
0 & I%
\end{array}%
\right) .$ Furthermore, if we define for $g\in U(\mathcal{A)}$,%
\begin{eqnarray*}
u(g) &=&\left( 
\begin{array}{cc}
g & -\delta (g) \\ 
0 & g%
\end{array}%
\right) :H\oplus _{2}H\rightarrow H\oplus _{2}H, \\
\tilde{G} &=&\left( 
\begin{array}{cc}
\tilde{g} & 0 \\ 
0 & \tilde{g}%
\end{array}%
\right) :X\rightarrow X,
\end{eqnarray*}%
we get $u(g)=S^{-1}\tilde{G}S$ and hence $\tilde{G}_{\mid
Y}=Su(g)S^{-1}:Y\rightarrow Y$ for each unitary $g$. Thus we obtain a
homomorphism%
\begin{equation*}
\mathcal{A}\ni a\rightarrow Su(a)S^{-1}\in B(Y)
\end{equation*}%
from $\mathcal{A}$ into $B(Y)$ that extends the group homomorphism 
\begin{equation*}
U(\mathcal{A)}\ni g\rightarrow \tilde{G}_{\mid Y}\in B(Y),
\end{equation*}%
and complete boundedness of this homomorphism is equivalent to innerity of
the derivation.
\end{remark}

\end{document}